\documentclass[12pt,oneside]{amsart}
\usepackage{amssymb}
\usepackage{amsfonts}
\usepackage{amscd}
\usepackage[matrix, arrow, curve]{xy}
\usepackage{graphicx}
\usepackage{color}
\usepackage{array}
\usepackage[matrix, arrow, curve]{xy}

\numberwithin{equation}{section}
\newtheorem{theorem}{Theorem}[section]

\newtheorem{corollary}[theorem]{Corollary}
\newtheorem{lemma}[theorem]{Lemma}
\theoremstyle{definition}

\theoremstyle{remark}
\newtheorem{remark}[theorem]{Remark}

\begin{document}
%%%%%%%%%%%%%%%%%%%%%%%%%%%%%%%%%%%%%%%%%%%%%%%%%
%%%%%%%%%%%%  macrodefinitions
%%%%%%%%%%%%%%%%%%%%%%%%%%%%%%%%%%%%%%%%%%%%%%%%%
%  Macros (general)
%%%%%%%%%%%%%%%%%%%%%%%
%\newcommand{\MgNekp}{\mathcal{M}_{g,N+1}^{(k,p)}} %% moduli space
%\newcommand{\M}{\mathcal{M}_{g,N+1}^{(1)}}
\newcommand{\M}{\mathcal{M}}
\newcommand{\F}{\mathcal{F}}

\newcommand{\Teich}{\mathcal{T}_{g,N+1}^{(1)}}
\newcommand{\T}{\mathrm{T}}
%%%%   temporary
\newcommand{\corr}{\bf}
\newcommand{\vac}{|0\rangle}
\newcommand{\Ga}{\Gamma}
\newcommand{\new}{\bf}
\newcommand{\define}{\def}
\newcommand{\redefine}{\def}
\newcommand{\Cal}[1]{\mathcal{#1}}
\renewcommand{\frak}[1]{\mathfrak{{#1}}}
\newcommand{\Hom}{\rm{Hom}\,}
%%%%%%%%%%%%%%%%%%%%%%%%%%%%%%%%%%%%
%   Referencing Scheme of Martin
%%%%%%%%%%%%%%%%%%%%%%%%%%
\newcommand{\refE}[1]{(\ref{E:#1})}
\newcommand{\refCh}[1]{Chapter~\ref{Ch:#1}}
\newcommand{\refS}[1]{Section~\ref{S:#1}}
\newcommand{\refSS}[1]{Section~\ref{SS:#1}}
\newcommand{\refT}[1]{Theorem~\ref{T:#1}}
\newcommand{\refO}[1]{Observation~\ref{O:#1}}
\newcommand{\refR}[1]{Remark~\ref{R:#1}}
\newcommand{\refP}[1]{Proposition~\ref{P:#1}}
\newcommand{\refD}[1]{Definition~\ref{D:#1}}
\newcommand{\refC}[1]{Corollary~\ref{C:#1}}
\newcommand{\refL}[1]{Lemma~\ref{L:#1}}
\newcommand{\refEx}[1]{Example~\ref{Ex:#1}}
\newcommand{\ovl}[1]{\overline{#1}}
\newcommand{\til}[1]{\widetilde{#1}}
\newcommand{\wht}[1]{\widehat{#1}}
%%%%%%%%%%%%%%%%%%%%%%%%%%%%%%%%%%
\newcommand{\R}{\ensuremath{\mathbb{R}}}
\newcommand{\C}{\ensuremath{\mathbb{C}}}
\newcommand{\N}{\ensuremath{\mathbb{N}}}
\newcommand{\Q}{\ensuremath{\mathbb{Q}}}
\renewcommand{\P}{\ensuremath{\mathcal{P}}}
\newcommand{\Z}{\ensuremath{\mathbb{Z}}}
%%%%%%%%%%%%%%%%%%%%%%%%%%%%%%%%%%%%%%%%%%
\newcommand{\kv}{{k^{\vee}}}
%%%%%%%%%%%%%%%%%%%%%%%%%%%%%%%%%%%%%%%%%%%%%
\renewcommand{\l}{\lambda}
%%%%%%%%%%%%%%%%%%%%%%%%%%%%%%%%%%%%%%%%%%%%%%%%%%
\newcommand{\gb}{\overline{\mathfrak{g}}}
\newcommand{\dt}{\tilde d}     % Oleg
\newcommand{\hb}{\overline{\mathfrak{h}}}
\newcommand{\g}{\mathfrak{g}}
\newcommand{\h}{\mathfrak{h}}
\newcommand{\gh}{\widehat{\mathfrak{g}}}
\newcommand{\ghN}{\widehat{\mathfrak{g}_{(N)}}}
\newcommand{\gbN}{\overline{\mathfrak{g}_{(N)}}}
\newcommand{\tr}{\mathrm{tr}}
\newcommand{\gln}{\mathfrak{gl}(n)}
\newcommand{\son}{\mathfrak{so}(n)}
\newcommand{\spnn}{\mathfrak{sp}(2n)}
\newcommand{\sln}{\mathfrak{sl}}
\newcommand{\sn}{\mathfrak{s}}
\newcommand{\so}{\mathfrak{so}}
\newcommand{\spn}{\mathfrak{sp}}
\newcommand{\tsp}{\mathfrak{tsp}(2n)}
\newcommand{\gl}{\mathfrak{gl}}
\newcommand{\slnb}{{\overline{\mathfrak{sl}}}}
\newcommand{\snb}{{\overline{\mathfrak{s}}}}
\newcommand{\sob}{{\overline{\mathfrak{so}}}}
\newcommand{\spnb}{{\overline{\mathfrak{sp}}}}
\newcommand{\glb}{{\overline{\mathfrak{gl}}}}
\newcommand{\Hwft}{\mathcal{H}_{F,\tau}}
\newcommand{\Hwftm}{\mathcal{H}_{F,\tau}^{(m)}}

%%%%%%%%%%%%%%%%%%%%%%%%%%%%%%%%%%%%%%%%%%%%%%%%%%%%
\newcommand{\car}{{\mathfrak{h}}}    % Cartan subalgebra
\newcommand{\bor}{{\mathfrak{b}}}    % Borel subalgebra
\newcommand{\nil}{{\mathfrak{n}}}    % nilpotent subalgebra
\newcommand{\vp}{{\varphi}}
\newcommand{\bh}{\widehat{\mathfrak{b}}}  % Borel subalgebra of KN algebra
\newcommand{\bb}{\overline{\mathfrak{b}}}  % Borel subalgebra of KN algebra
\newcommand{\Vh}{\widehat{\mathcal V}}
\newcommand{\KZ}{Kniz\-hnik-Zamo\-lod\-chi\-kov}
\newcommand{\TUY}{Tsuchia, Ueno  and Yamada}
\newcommand{\KN} {Kri\-che\-ver-Novi\-kov}
\newcommand{\pN}{\ensuremath{(P_1,P_2,\ldots,P_N)}}
\newcommand{\xN}{\ensuremath{(\xi_1,\xi_2,\ldots,\xi_N)}}
\newcommand{\lN}{\ensuremath{(\lambda_1,\lambda_2,\ldots,\lambda_N)}}
\newcommand{\iN}{\ensuremath{1,\ldots, N}}
\newcommand{\iNf}{\ensuremath{1,\ldots, N,\infty}}

\newcommand{\tb}{\tilde \beta}
\newcommand{\tk}{\tilde \varkappa}
\newcommand{\ka}{\kappa}
\renewcommand{\k}{\varkappa}
\newcommand{\ce}{{c}}

\newcommand{\Pif} {P_{\infty}}
\newcommand{\Pinf} {P_{\infty}}
\newcommand{\PN}{\ensuremath{\{P_1,P_2,\ldots,P_N\}}}
\newcommand{\PNi}{\ensuremath{\{P_1,P_2,\ldots,P_N,P_\infty\}}}
\newcommand{\Fln}[1][n]{F_{#1}^\lambda}
\newcommand{\tang}{\mathrm{T}}
\newcommand{\Kl}[1][\lambda]{\can^{#1}}
\newcommand{\A}{\mathcal{A}}
\newcommand{\U}{\mathcal{U}}
\newcommand{\V}{\mathcal{V}}
\newcommand{\W}{\mathcal{W}}
\renewcommand{\O}{\mathcal{O}}
\newcommand{\Ae}{\widehat{\mathcal{A}}}
\newcommand{\Ah}{\widehat{\mathcal{A}}}
\newcommand{\La}{\mathcal{L}}
\newcommand{\Le}{\widehat{\mathcal{L}}}
\newcommand{\Lh}{\widehat{\mathcal{L}}}
\newcommand{\eh}{\widehat{e}}
\newcommand{\Da}{\mathcal{D}}
\newcommand{\kndual}[2]{\langle #1,#2\rangle}
\newcommand{\cins}{\frac 1{2\pi\mathrm{i}}\int_{C_S}}
\newcommand{\cinsl}{\frac 1{24\pi\mathrm{i}}\int_{C_S}}
\newcommand{\cinc}[1]{\frac 1{2\pi\mathrm{i}}\int_{#1}}
\newcommand{\cintl}[1]{\frac 1{24\pi\mathrm{i}}\int_{#1 }}
\newcommand{\w}{\omega}
\newcommand{\ord}{\operatorname{ord}}
\newcommand{\res}{\operatorname{res}}
\newcommand{\nord}[1]{:\mkern-5mu{#1}\mkern-5mu:}
\newcommand{\codim}{\operatorname{codim}}
\newcommand{\ad}{\operatorname{ad}}
\newcommand{\Ad}{\operatorname{Ad}}
\newcommand{\supp}{\operatorname{supp}}

%%%%%%%%%%%%%%%%%%%%%%%%%%%%%%%%%%%%%%%%%%%%%%%%
\newcommand{\Fn}[1][\lambda]{\mathcal{F}^{#1}}
\newcommand{\Fl}[1][\lambda]{\mathcal{F}^{#1}}
\renewcommand{\Re}{\mathrm{Re}\,}
\renewcommand{\Im}{\mathrm{Im}\,}

\newcommand{\ha}{H^\alpha}

\define\ldot{\hskip 1pt.\hskip 1pt}
\define\ifft{\qquad\text{if and only if}\qquad}
\define\a{\alpha}
\redefine\d{\delta}
\define\w{\omega}
\define\ep{\epsilon}
\redefine\b{\beta} \redefine\t{\tau}
\redefine\i{{\mathrm{i}}}
\define\ga{\gamma}
\define\cint #1{\frac 1{2\pi\i}\int_{C_{#1}}}
\define\cintta{\frac 1{2\pi\i}\int_{C_{\tau}}}
\define\cintt{\frac 1{2\pi\i}\oint_{C}}
\define\cinttp{\frac 1{2\pi\i}\int_{C_{\tau'}}}
\define\cinto{\frac 1{2\pi\i}\int_{C_{0}}}
%\define\cinttt{\frac 1{24\pi\i}\int_{C_{\tau}}}
\define\cinttt{\frac 1{24\pi\i}\int_C}
\define\cintd{\frac 1{(2\pi \i)^2}\iint\limits_{C_{\tau}\,C_{\tau'}}}
\define\dintd{\frac 1{(2\pi \i)^2}\iint\limits_{C\,C'}}
\define\cintdr{\frac 1{(2\pi \i)^3}\int_{C_{\tau}}\int_{C_{\tau'}}
\int_{C_{\tau''}}}
\define\im{\operatorname{Im}}
\define\re{\operatorname{Re}}
%\define\res{\text{res}}
\define\res{\operatorname{res}}
\redefine\deg{\operatornamewithlimits{deg}}
\define\ord{\operatorname{ord}}
\define\rank{\operatorname{rank}}
\define\fpz{\frac {d }{dz}}
\define\dzl{\,{dz}^\l}
\define\pfz#1{\frac {d#1}{dz}}

\define\K{\Cal K}
\define\U{\Cal U}
\redefine\O{\Cal O}
\define\He{\text{\rm H}^1}
\redefine\H{{\mathrm{H}}}
\define\Ho{\text{\rm H}^0}
\define\A{\Cal A}
\define\Do{\Cal D^{1}}
\define\Dh{\widehat{\mathcal{D}}^{1}}
\redefine\L{\Cal L}
\newcommand{\ND}{\ensuremath{\mathcal{N}^D}}
\redefine\D{\Cal D^{1}}
\define\KN {Kri\-che\-ver-Novi\-kov}
\define\Pif {{P_{\infty}}}
\define\Uif {{U_{\infty}}}
\define\Uifs {{U_{\infty}^*}}
\define\KM {Kac-Moody}
\define\Fln{\Cal F^\lambda_n}
%%%%%%%%%%%%%%%%%%%%
\define\gb{\overline{\mathfrak{ g}}}
\define\G{\overline{\mathfrak{ g}}}
\define\Gb{\overline{\mathfrak{ g}}}
\redefine\g{\mathfrak{ g}}
\define\Gh{\widehat{\mathfrak{ g}}}
\define\gh{\widehat{\mathfrak{ g}}}
%%%%%%%%%%%%%%%%%%%%%%%%%%
\define\Ah{\widehat{\Cal A}}
\define\Lh{\widehat{\Cal L}}
\define\Ugh{\Cal U(\Gh)}
\define\Xh{\hat X}
\define\Tld{...}
\define\iN{i=1,\ldots,N}
\define\iNi{i=1,\ldots,N,\infty}
\define\pN{p=1,\ldots,N}
\define\pNi{p=1,\ldots,N,\infty}
\define\de{\delta}

\define\kndual#1#2{\langle #1,#2\rangle}
\define \nord #1{:\mkern-5mu{#1}\mkern-5mu:}
%\define \MgN{{\Cal M}_{g,N}} %% moduli space
%\define \MgNp{{\Cal M}_{g,N}^{(p)}} %% moduli space
\newcommand{\MgN}{\mathcal{M}_{g,N}} %% moduli space
\newcommand{\MgNeki}{\mathcal{M}_{g,N+1}^{(k,\infty)}} %% moduli space
\newcommand{\MgNeei}{\mathcal{M}_{g,N+1}^{(1,\infty)}} %% moduli space
\newcommand{\MgNekp}{\mathcal{M}_{g,N+1}^{(k,p)}} %% moduli space
\newcommand{\MgNkp}{\mathcal{M}_{g,N}^{(k,p)}} %% moduli space
\newcommand{\MgNk}{\mathcal{M}_{g,N}^{(k)}} %% moduli space
\newcommand{\MgNekpp}{\mathcal{M}_{g,N+1}^{(k,p')}} %% moduli space
\newcommand{\MgNekkpp}{\mathcal{M}_{g,N+1}^{(k',p')}} %% moduli space
\newcommand{\MgNezp}{\mathcal{M}_{g,N+1}^{(0,p)}} %% moduli space
\newcommand{\MgNeep}{\mathcal{M}_{g,N+1}^{(1,p)}} %% moduli space
\newcommand{\MgNeee}{\mathcal{M}_{g,N+1}^{(1,1)}} %% moduli space
\newcommand{\MgNeez}{\mathcal{M}_{g,N+1}^{(1,0)}} %% moduli space
\newcommand{\MgNezz}{\mathcal{M}_{g,N+1}^{(0,0)}} %% moduli space
\newcommand{\MgNi}{\mathcal{M}_{g,N}^{\infty}} %% moduli space
\newcommand{\MgNe}{\mathcal{M}_{g,N+1}} %% moduli space
\newcommand{\MgNep}{\mathcal{M}_{g,N+1}^{(1)}} %% moduli space
\newcommand{\MgNp}{\mathcal{M}_{g,N}^{(1)}} %% moduli space
\newcommand{\Mgep}{\mathcal{M}_{g,1}^{(p)}} %% moduli space
\newcommand{\MegN}{\mathcal{M}_{g,N+1}^{(1)}} %% moduli space

%\define \mpt{(M,P_1,P_2,\ldots, P_N,\Pif)} %% moduli point
%\define \mpp{(M,P_1,P_2,\ldots, P_N)} %% moduli point
%\define \MgNn{{\Cal M}_{g,N}^{(1)}} %% moduli space
%\define \MgNen{{\Cal M}_{g,N+1}^{(1)}} %% moduli space
%\define \Mgo{{\Cal M}_{g,0}} %% moduli space
%\define \mptn{(M,P_1,P_2,\ldots, P_N,\Pif,z_1,\ldots,z_N,z_\infty)}
 %% moduli point
%\define \mppn{(M,P_1,P_2,\ldots, P_N,z_1,\ldots,z_N)} %% moduli point
\define \sinf{{\widehat{\sigma}}_\infty}
\define\Wt{\widetilde{W}}
\define\St{\widetilde{S}}
\newcommand{\SigmaT}{\widetilde{\Sigma}}
\newcommand{\hT}{\widetilde{\frak h}}
\define\Wn{W^{(1)}}
\define\Wtn{\widetilde{W}^{(1)}}
\define\btn{\tilde b^{(1)}}
\define\bt{\tilde b}
\define\bn{b^{(1)}}
\define \ainf{{\frak a}_\infty} %matrices with a finite number of
                                %diagonals

%
%%%%%%%%%% Olegs definitions %%%%%%%%%%%%%%%%%%%%%%%%%%%%%%%%%%%
\define\eps{\varepsilon}    % Oleg
\newcommand{\e}{\varepsilon}
\define\doint{({\frac 1{2\pi\i}})^2\oint\limits _{C_0}
       \oint\limits _{C_0}}                            % Oleg
\define\noint{ {\frac 1{2\pi\i}} \oint}   % Oleg
\define \fh{{\frak h}}     % Oleg
\define \fg{{\frak g}}     % Oleg
\define \GKN{{\Cal G}}   % affine Krichever-Novikov algebra % Oleg
\define \gaff{{\hat\frak g}}   % affine Krichever-Novikov algebra
\define\V{\Cal V}
\define \ms{{\Cal M}_{g,N}} %% moduli space
\define \mse{{\Cal M}_{g,N+1}} %% moduli space
%%%%%%%%%%%%%%%%%%%%%%%%%%%%%%%%%%%%%%
\define \tOmega{\Tilde\Omega}
\define \tw{\Tilde\omega}
\define \hw{\hat\omega}
\define \s{\sigma}
\define \car{{\frak h}}    % Cartan subalgebra
\define \bor{{\frak b}}    % Borel subalgebra
\define \nil{{\frak n}}    % nilpotent subalgebra
\define \vp{{\varphi}}
\define\bh{\widehat{\frak b}}  % Borel subalgebra of KN algebra
\define\bb{\overline{\frak b}}  % Borel subalgebra of KN algebra
\define\KZ{Knizhnik-Zamolodchikov}
\define\ai{{\alpha(i)}}
\define\ak{{\alpha(k)}}
\define\aj{{\alpha(j)}}
\newcommand{\calF}{{\mathcal F}}
\newcommand{\ferm}{{\mathcal F}^{\infty /2}}
\newcommand{\Aut}{\operatorname{Aut}}
\newcommand{\End}{\operatorname{End}}
%%%%%%%%%%%%%%%%%%%%%%%%%%%%%%%%%%%%%%%%%%%
%%%%%%%%%%%%%%%%  цвет %%%%%%%%%%%%%%%%%%%%%%%%%%%%%%%%%
\newcommand{\red}{\color[rgb]{1,0,0}}
\newcommand{\blue}{\color[rgb]{0,0,1}}
\newcommand{\viol}{\color[rgb]{1,0,1}}%%%%%%%%%%%%%%%%%%%%%%%%%%%%%%%%%%%%%%%%%%%%%%

%%%%%%%%%%%%%%%%%%%%%%%%%%%%%%%%%
%%%%%%%%%%%%%%   for laxcent
%%%%%%%%%%%%%%%%%%%%%%%%%%%%%%%%%%
\newcommand{\laxgl}{\overline{\mathfrak{gl}}}
\newcommand{\laxsl}{\overline{\mathfrak{sl}}}
\newcommand{\laxso}{\overline{\mathfrak{so}}}
\newcommand{\laxsp}{\overline{\mathfrak{sp}}}
\newcommand{\laxs}{\overline{\mathfrak{s}}}
\newcommand{\laxg}{\overline{\frak g}}
\newcommand{\bgl}{\laxgl(n)}
%%%%%%%%%%%%%%%%%%%%%%%%
\newcommand{\tX}{\widetilde{X}}
\newcommand{\tY}{\widetilde{Y}}
\newcommand{\tZ}{\widetilde{Z}}
%%%%%%%%%%%%%%%%%%%%%%%%%%%%%%%%%%%%%%%%%%
%%%%%%%%%%%%  END of macrodefinitions
%%%%%%%%%%%%%%%%%%%%%%%%%%%%%%%%%%%%%%%%%

%%%%%%%%%%%%%%%%%%%%%%%%%%%%%%%%%
%Top-Matter
%%%%%%%%%%%%%%%%%%%%%%%%%%%%%%%
%%%%%%%%%%%%%%%%%    private header  %%%%%%%%%%%%%%%%%%%%

%\large{
\title[]{Inversion of the Abel--Prym map for real curves with involutions}
\author[O.K.Sheinman]{O.K.Sheinman}
%\date{\today}
\thanks{..............................................}
\address{Steklov Mathematical Institute of the Russian Academy of Sciences}
\dedicatory{}
\maketitle
\begin{abstract}
Riemann vanishing theorem is a main ingredient of the conventional technique related to the Jacobi inversion problem. In the case of curves with a holomorphic involution, it has been presented quite fully in wellknown Fay's Lectures on theta functions. The case of real algebraic curves with involution is presented with less completeness in the literature. We provide a detailed presentation of that case, including the case of real curves of the non-separating type with a holomorphic involution, not considered before with this relation. In particular, we formulate the symmetry of the Prym theta function in this case.

Bibliography: 17 titles.

\end{abstract}
\tableofcontents
%%%%%%%%%%%%%%%%%%%%%%%%%%%%%%%%%%%%%%%%
\section{Introduction}
Let $\Sigma$ be a genus $g$ agebraic curve. It is classical that the Abel map establishes a birational (bimeromorphic on certain open dense subsets) correspondence between ${\rm Sym}^g\Sigma$ and the Jacobian $Jac(\Sigma)$ of $\Sigma$. Also, it can be considered as a correspondence between equivalence classes of degree $g$ divisors and points of $Jac(\Sigma)$. The inverse to the Abel map is given by the following Riemann theorem: the preimage of a point of the Jacobian (except for the points of a subvariety of codimension 1) is given by a zero divisor of a certain auxiliary function (constructed using the Riemann $\theta$-function) on the universal covering of $\Sigma$. The last theorem is referred to as the Riemann vanishing theorem below, and the above reversion procedure as a whole is referred to as the Jacobi inversion. Among other things, the Jacobi inversion is a powerful tool of the theory of integrable systems.

If $\Sigma$ is endowed with a holomorphic involution (denoted by $\s$ below) then $\Sigma$ can be assigned with another Abelian variety called Prym variety (or Prymian), defined as the subset of $\s$-antiinvariant points in $Jac(\Sigma)$, and denoted by $Prym_\s(\Sigma)$ below. However there is a more primary object in relation to the Prymian, namely, its finite unramified covering called isoPrymian below ($isoPrym(\Sigma)$). The construction of $isoPrym(\Sigma)$ repeats the construction of $Jac(\Sigma)$ where the Riemann matrix is replaced with the Prym matrix. In particular, isoPrymian is always principally polarized. The corresponding analog of the Abel map using only $\s$-antiinvariant holomorphic differentials is called the Abel--Prym map (see \refSS{Ab_Pr} for definitions).

As we noted in \cite{Sh_Bin}, the Jacobi inversion problem should be modified if applied to the Abel--Prym map. First, it is shown below that it should be posed on the isoPrymian. Second, the number of points returned by the Riemann vanishing theorem is twice as big as $\dim isoPrym(\Sigma)$ \cite{Fay} but the divisor~$\zeta$ formed by them satisfies the relation $\zeta+\s \zeta\sim D$ where $D$ is a constant divisor. Thus in this case, the variety of divisors in the Riemann theorem is given as an $h$-dimensional subvariety in $\C^{2h}$ where $h=\dim isoPrym(\Sigma)$, wich is less effective than in the classical case,  in which these divisors are free of any relation. A more effective description is also available in the particular case of two commuting involutions \cite{Sh_PrymJacobi}. In the present paper we consider a general case. In particular, we show that a direct calculation of symmetric functions of the points in the divisor is still possible in this case, which is a weakened approach to the Jacobi inversion problem due to Dubrovin \cite{Dubr_theta} (going back to Riemann).

If $\Sigma$ possesses an antiholomorphic involution, it is called a real curve.  Jacobians of real curves have been first investigated in \cite{Dubr_Nat} with relation to real solutions to sin-Gordon equation. It was shown that the Abel map  establishes a 1-to-1 correspondence between $\tau$-invariant degree $g$ divisors and the real part of the Jacobian $Jac(\Sigma)$.  Investigation of Prymians of real curves was pioneered in \cite{VN} with relation to real solutions to the potential two-dimensional Schr\"{o}dinger equation, and also presented in \cite{Nat1} in the case when $\s$ has only two fixed points. A general result of \cite{VN} is that a certain shift of the real part of such Prymian is in a 1-to-1 correspondence with divisors of a certain degree on $\Sigma$ satisfying the relations $\zeta+\s \zeta\sim D$ and $\tau\zeta=\zeta$, and the correspondence is established by the Abel transform (see \refL{inver1}.$1^\circ$-$3^\circ$). Together, \refL{inver1}.$4^\circ$ and \refT{invers} give the inversion theorem for the Abel--Prym map of real curves. It claims that the inverse image of a certain real subvariety of isoPrymian under the Abel--Prym map is given by $\tau$-invariant, or $\s\tau$-invariant divisors $\zeta$ on $\Sigma$ satisfying the relation $\zeta+\s\zeta=D$ where $D$ is a constant divisor. Proof of the theorem is based on the results of \cite{Fay,Sh_PrymJacobi}, and on the study of symmetries of the Prym $\theta$-function (\refL{P_symm} of the present paper) for both separating and non-separating real curves. It generalizes the results of \cite{Dubr_Nat,Dubr_ItNa} on the Riemann $\theta$-function of a real curve, and of \cite{Fay} on the Prym $\theta$-function of a real curve of separating type.

In \refS{real_Jac}, following  \cite{Dubr_Nat,Dubr_ItNa,Nat1}, we give preliminaries on Jacobians and Riemann $\theta$-functions of real curves. In \refS{ogr} we introduce main notions related to Prym varieties, study the realness properties of Prym matrices and symmetries of Prym $\theta$-functions of real curves. In \refS{inv_th} we present our inversion theorem for real curves.

%%%%%%%%%%%%%%%%%%%%%%%%%%%%%%%%%%%%%%%%%%%%%%%

%%%%%%%%%%%%%%%%%%%%%%%%%%%%%%%%%%%%%%%%%%%%%%%
\section{Real curves, their Jacobians and $\theta$-functions}\label{S:real_Jac}
This section contains a background information, and follows the lines of \cite{Dubr_Nat,Nat1,Nat2,Nat3}

%%%%%%%%%%%%%%%%%%%%%%%%%%%%%%%%%%%%%%%%%%%%%%
\subsection{Separating and non-separating curves. Topological type. Real bases}
Let $\Sigma$ be a compact algebraic curve over $\C$. If $\Sigma$ possesses an antiholomorphic involution (antiinvolution for short) then it is referred to as a real curve. Let $\tau$ stay for such involution. Connected components of the set of fixed points of $\tau$ are called ovals. If $\Sigma\setminus\cup\,{ovals}$ is not connected, the pair $(\Sigma,\tau)$ is called a real curve of separating type (separating curve), otherwise it is called a non-separating curve. For separating curves, the number of connected components of $\Sigma\setminus\cup\,{ovals}$ is always equal to 2 \cite{Nat1}.

Two real curves $(\Sigma_1,\tau_1)$ and $(\Sigma_2,\tau_2)$ are called topologically equivalent if there is a homeomorphism $\psi :\Sigma_1\to\Sigma_2$ such that $\psi\tau_1=\tau_2\psi$. Set $\eps=1$ for separating curves, and $\eps=0$ for non-separating curves. Let $g$ stay for the genus of $\Sigma$, and $k$ stay for the number of ovals. The set $(g,k,\eps)$ is called the topological type of the pair $(\Sigma,\tau)$.
\begin{theorem}[\cite{Nat1}]\label{T:ttype}
Two real curves are topologically equivalent iff they have the same topological type.
\end{theorem}
For more information on the structure of real curves we refer to \cite{Nat1}, and the works quoted there.
\begin{theorem}[\cite{Nat1}]\label{T:r_base}
Let $(\Sigma,\tau)$ be a real curve of the type $(g,k,\eps)$, and $q\in\Sigma$ is a real point. Then there exists a symplectic base $\{ a_j,b_j|j=1,\ldots,g\}$ of cycles on $\Sigma$ such that
\begin{itemize}
  \item[$1^\circ$] for $\eps=1$
\[
           \left\{
              \begin{array}{ll}
                  \tau(a_i)=a_i, \tau(b_i)=-b_i, & i=1,\ldots,k-1; \\
                \tau(a_i)=a_{i+m}, \tau(b_i)=-b_{i+m}, & i=k,\ldots,k+m-1
              \end{array}
            \right.
\]
where $m=\frac{1}{2}(g+1-k)$, and the oval containing $q$ is homological to $\sum_{i=1}^{k-1}a_i$;
  \item[$2^\circ$] for $\eps=0$
\[
           \left\{
              \begin{array}{ll}
                  \tau(a_i)=a_i,  & i=1,\ldots,g; \\
                  \tau(b_i)=-b_i, & i=1,\ldots,k-1; \\
                  \tau(b_i)=-b_i+a_i, & i=k,\ldots,g, \\
              \end{array}
            \right.
\]
and the oval containing $q$ is homological to $\sum_{i=1}^{g}a_i$.
\end{itemize}
\end{theorem}
A base satisfying to the conditions of \refT{r_base} is called a real base.
%%%%%%%%%%%%%%%%%%%%%%%%%%%%%%%%%%%%%%%%%%%%%%
\subsection{Realness properties of the Riemann matrix of a real curve}\label{SS:realJac}
Let $\{ a_j,b_j\}$ be a real base of cycles of a curve, $\{\w_i\}$ be the normalized base of  holomorphic differentials where the normalization conditions are of the form
\begin{equation}\label{E:norm}
  \int_{a_j}\w_i=2\pi {\rm i}\d_{ij}.
\end{equation}
Define the permutation $t$  of indices $1,\ldots,g$, such that $\tau(a_j)=a_{t(j)}$ by virtue of \refT{r_base}. For separating curves $t$ writes as follows via cyclic permutations: $t=(1)\ldots(k-1)(k,k+m)\ldots(k+m-1,g)$.
For non-separating curves $t$ is trivial. Observe that $t^2=1$.

Let $\w=(\w_1,\ldots,\w_g)^T$ be the column of normalized differentials, $t\w=(\w_{t(1)},\ldots,\w_{t(g)})^T$, $A_i=\int_{a_i}\w$, $B_i=\int_{b_i}\w$ be the corresponding periods. Then $(A_1,\ldots,A_g)=2\pi{\rm i}E$, $(B_1,\ldots,B_g)=B$, where $E$ is the unit matrix, $B$ is referred to as the matrix of periods. It is a symmetric matrix with negative defined real part.

Below we describe specific properties of the period matrix, and symmetries of the Riemann $\theta$-function for real curves.

\begin{lemma}\label{L:tau^*}
\begin{itemize}
 \item[$1^\circ$.]
$\tau^*\w=-t\ovl{\w}$;
  \item[$2^\circ$.]
For separating real curves $\tau^*\w_i=\left\{
                                  \begin{array}{ll}
                                    -\ovl{\w_i}, & \hbox{$i<k$;} \\
                                    -\ovl{\w_{i+m}}, & \hbox{$k\le i<k+m$.}
                                         \end{array}
                                       \right.
                          $
  \item[$3^\circ$.]

For non-separating curves $\tau^*\w = -\overline{\w}$.
\end{itemize}
\end{lemma}
\begin{proof}
By change of variables $\int_{a_j}\tau^*\w_i = \int_{a_{t(j)}}\w_i$. By symmetry of the $a$-period matrix (it is just diagonal) $\int_{a_{t(j)}}\w_i = \int_{a_j}\w_{t(i)}$. Besides, the matrix of $a$-periods is imaginary, hence
$\int_{a_j}\w_{t(i)} = -\int_{a_j}\ovl{\w_{t(i)}}$. Both $\tau^*\w_i$ and $-\ovl{\w_{t(i)}}$ are antiholomorphic differentials, and they have the same $a$-periods. Hence $\tau^*\w_i = -\ovl{\w_{t(i)}}$, $i=1,\ldots,g$, and $1^\circ$ is proven. $2^\circ$ and $3^\circ$ immediately follow from $1^\circ$ by definition of the permutation~$t$.
\end{proof}
\begin{lemma}\label{L:8.2}
Let $\{ a_j,b_j\}$ be a real base of cycles of a real curve of the type $(g,k,\varepsilon)$, then
\begin{itemize}
\item[$1^\circ$.]
$\overline{B_j}=B_j$ for $j\le k-1$;
\item[$2^\circ$.]
$\overline{B_j}=B_j-A_j$ for $\varepsilon=0$, $j= k,\ldots,g$;
\item[$3^\circ$.]
$\overline{B_j}=tB_{j+m}$ for $\varepsilon=1$, $j= k,\ldots,k+m-1$, where $m=\frac{1}{2}(g+1-k)$.
\end{itemize}
\end{lemma}
\begin{remark}
For separating curves, symmetries of $B$ given by the cases $1^\circ$, $3^\circ$ of the lemma are the same as claimed in \cite[p.109]{Fay}. They are also formulated in Lemma 8.2 \cite{Nat1}, however without mentioning the permutation $t$ in the third relation. For hyperelliptic non-separating curves, the lemma was first formulated in \cite{Dubr_Nat}, and then for general non-separating curves as Lemma 8.2 \cite{Nat1}.
\end{remark}
\begin{proof}[Proof of \refL{8.2}]
The proof in the cases $1^\circ$ and $3^\circ$ immediately follows from the relation $\ovl{B_j}=tB_{t(j)}$. Let's check the last: $\ovl{B_j}= \ovl{\int_{b_j}\w}  = \int_{b_j}\ovl{\w} = -t\int_{b_j}\tau^*{\w} = -t\int_{\tau(b_j)}\w = -t\int_{-b_{t(j)}}\w = tB_{t(j)}$.

In the case $2^\circ$ $t$ is trivial. Similar to the cases $1^\circ$ and $3^\circ$ we get $\ovl{B_j}= -\int_{\tau(b_j)}\w$. We proceed as follows: $\int_{\tau(b_j)}\w = \int_{-b_j+a_j}\w = -B_j+A_j$.
\end{proof}
%%%%%%%%% вставка с накрытием - начало
%The involution $\tau$ can be lifted as $\widetilde{\tau}$ to the universal covering $\widetilde{\Sigma}$ of $\Sigma$. Indeed, we regard $\widetilde{\Sigma}$ as the space of homotopy classes of paths on $\Sigma$ with a fixed beginning $q$ and given ends, and then operate on paths by $\tau$ pointwise. We continue $\tau$ ($\widetilde{\tau}$, resp.) onto ${\rm Sym}^g\Sigma$ (${\rm Sym}^g\widetilde{\Sigma}$, resp.).
%\begin{equation}\label{E:diagr}
%\xymatrix{ \C^g \ar[r]^{-\tau_\R}   & \C^g  \\
%           {\rm Sym}^g\widetilde{\Sigma}\ar[u]^i\ar[r]^{\widetilde{\tau}} \ar[d]_{pr}& {\rm Sym}^g\widetilde{\Sigma}\ar[u]_i \ar[d]_{pr}   \\
%{\rm Sym}^g\Sigma \ar[r]^\tau & {\rm Sym}^g\Sigma
%}
%\end{equation}
%Next, we embed ${\rm Sym}^g\widetilde{\Sigma}$ in $\C^g$ by assigning every unordered set $\{(\ga_1,P_1),\ldots,(\ga_g,P_g)\}$ with $z=\sum_{j=1}^g\int_{\ga_j}\w\in\C^g$ where the path $\ga_j$ connects $q$ with~$P_j\in\Sigma$ (the injection $i$ at the diagram \refE{diagr}).
%%%%%%%%% вставка с накрытием - конец
Using the Abel transform we transfer the antiinvolution $\tau$ to $Jac(\Sigma)$. Indeed, by the Riemann vanishing theorem we uniquely represent almost every $z\in Jac(\Sigma)$ as $z=\int^{D}\w$ where $D$ is a degree $g$ divisor on $\Sigma$. Then we set by definition $\tau(z)=\int^{D}\tau^*\w$. To prove that $\tau(z)$ is well-defined, we check that the period lattice is $\tau$-invariant:
\begin{equation}\label{E:tau-inv}
  \tau(A_j)=A_{t(j)},\quad \tau(B_j)=-B_{t(j)}
\end{equation}
where the second is true modulo $a$-periods.
Indeed, $\tau(A_j)=\int_{a_j}\tau^*\w = -t\int_{a_j}\ovl{\w} = -t\ovl{A_j} = tA_j=A_{t(j)}$ (here we used imaginary and symmetry of the matrix of $a$-periods). Similarly, $\tau(B_j)=\int_{b_j}\tau^*\w = -t\ovl{B_j} = -t(tB_{t(j)}) = -B_{t(j)}$ (since \refL{8.2} implies that $\ovl{B_j}\equiv tB_{t(j)}\, ({\rm mod}A_j)$).

Observe that
\begin{equation}\label{E:tau_coord}
    \tau(z)=-t\ovl{z}, \quad z\in Jac(\Sigma).
\end{equation}
Indeed, $\tau(z) = \int^D\tau^*\w = -t\ovl{\int^D\w} = -t\ovl{z}$.

Next, we define an $\R$-linear involution $\til{\tau}:\,\C^g\to\C^g$ by
\begin{equation}\label{E:tau_R}
            \til{\tau}(z)=-t\ovl{z}, \quad z\in \C^g.
\end{equation}
Hence
\begin{equation}\label{E:tau_R_sep}
  \til{\tau} z=-t\ovl{z}=-(\overline{z_1},\ldots,\overline{z_{k-1}}, \overline{z_{k+m}},\ldots,\overline{z_g}, \overline{z_k},\ldots,\overline{z_{k+m-1}})^T.
\end{equation}
for separating curves, and
\begin{equation}\label{E:tau_R_non}
  \til{\tau} z=-\ovl{z}.
\end{equation}
for non-separating curves.
\begin{lemma}\label{L:til_tau}
The action of $\til{\tau}$ on the periods can be written in the following two equivalent forms:
\[
%\left\{
  \begin{array}{lll}
1^\circ . &
   \til{\tau} A_j=A_{t(j)},\quad \til{\tau} B_j=-B_{t(j)}, & \varepsilon=1,\ j=1,\ldots,g\ \text{or}\ \varepsilon=0,\ j=1,\ldots,k-1; \\
   & \til{\tau} A_j=A_j,\quad \til{\tau} B_j=-B_j+A_j, & \varepsilon=0,\ j=k,\ldots,g;  \\
2^\circ . & \til{\tau}{A_j}=A_j, \til{\tau}{B_j}=-B_j, & j=1,\ldots, k-1; \\
  & \til{\tau}{A_j}=A_j, \til{\tau}{B_j}=-B_j+A_j, & \varepsilon=0,\ j=k,\ldots,g;   \\
    & \til{\tau}{A_j}=A_{j+m}, \til{\tau}{B_j}=-B_{j+m}, & \hbox{$\varepsilon=1$, $j= k,\ldots,k+m-1$;} \\
    & \til{\tau}{A_j}=A_{j-m}, \til{\tau}{B_j}=-B_{j-m}, & \hbox{$\varepsilon=1$, $j= k+m,\ldots,g$}
  \end{array}
%\right.
\]
where $m=\frac{1}{2}(g+1-k)$.

\hskip-3pt
$3^\circ$. The following diagram is commutative:
\[
\xymatrix{
           \C^g\ar[r]^{\til{\tau}} \ar[d]_{pr}& \C^g \ar[d]_{pr}   \\
  Jac(\Sigma) \ar[r]^\tau & Jac(\Sigma)
}
\]
where $pr$ is the natural projection.
\end{lemma}
\begin{proof}
$1^\circ$. By \refE{tau_R} we have
\begin{align*}\label{E:tauR}
  & \til{\tau}(A_j)=-t\ovl{A_j} = -t(-A_j) = tA_j = A_{t(j)}\ \text{in all cases}; \\
  & \til{\tau}(B_j)=-t\ovl{B_j} = -t(tB_{t(j)}) = -B_{t(j)}\ \text{except for}\ \varepsilon=0, j=k,\ldots,g;\\
& \til{\tau}(B_j)=-\ovl{B_j}\ \text{(by \refE{tau_R_non})} = -B_j+A_j\ \text{for}\ \varepsilon=0, j=k,\ldots,g.
\end{align*}
In the last two lines we made use of \refL{8.2}.

Claim $2^\circ$ follows from $1^\circ$ by definition of $t$.

By claim $1^\circ$ it follows that the fundamental lattice of the Jacobian is $\til{\tau}$-invariant. Then \refE{tau_coord} and \refE{tau_R} prove $3^\circ$.
\end{proof}
\begin{remark}
  $-\til{\tau}$ coincides with the $\tau_\R$ introduced in \cite{Nat1} except for the case $2^\circ$, $\varepsilon=0$ of \refL{til_tau}.
\end{remark}

%%%%%%%%%%%%%%%%%%%%%%%%%%%%%%%%%%%%%%%%%%%%%%
\subsection{Riemann $\theta$-function of a real curve}
The Riemann theta function is defined by
\begin{equation}\label{E:thetaRiem}
  \theta(z,B)=\sum_{N\in \Z^g}\exp(\frac{1}{2}N^TBN+N^Tz).
\end{equation}
\begin{lemma}[\cite{Dubr_ItNa}]\label{L:symm}
The Riemann theta function of a real curve possesses the following symmetries:
\begin{itemize}
  \item[$1^\circ$]
$\theta(t\ovl{z})=\ovl{\theta(z)}$ for separating curves;
  \item[$2^\circ$]
$\ovl{\theta(z)}=\theta(\ovl{z}+\l)$ for non-separating curves, where
\begin{equation}\label{E:lambda}
\l= \pi \i (0,\ldots,0,1,\ldots,1)^T \ \ (\text{units for}\ j\ge k).
\end{equation}
\end{itemize}
\end{lemma}
\begin{proof}
$1^\circ$. \cite[proof of Prop. 6.1]{Fay}. By \refE{tau_R_non} (and for the reason that $\theta$ is even in $z$) $\theta(tz) = \sum_{N\in\Z^g}\exp(\frac{1}{2}N^TBN+N^T(t\ovl{z}))$. Since $t^2=1$ and $t^T=t$ we have $N^TBN=(tN)^T(tBt)(tN)$. By \refL{8.2} $\ovl{B}=tBt$ (see the proof of \refL{8.2}). By setting $M=tN$ we get $\theta(\til{\tau}z)=\sum_{M\in\Z^g}\exp(\frac{1}{2}M^T\ovl{B}M+M^T\ovl{z}) = \ovl{\theta(z)}$.

$2^\circ$. $\theta(\ovl{z}+\l)=\sum_{N\in \Z^g}\exp(\frac{1}{2}N^TBN+N^T\l+N^T\ovl{z})$. We want to show that
\begin{equation}\label{E:L8.2_2}
  \frac{1}{2}N^TBN+N^T\l\equiv \frac{1}{2}N^T(B-\widetilde{A})N ({\rm mod}\, 2\pi{\rm i}\Z),\quad \forall N\in\Z^g
\end{equation}
where $\widetilde{A}=(0,\ldots,0,A_k,\ldots,A_g)$, $A_j=2\pi{\rm i}(\d_j^i)^{i=1,\ldots,g}$, $j=1,\ldots,g$ are columns of the matrix of $a$-periods. According to \refL{8.2},$2^\circ$  $B-\widetilde{A}=\ovl{B}$, which proves the required symmetry.

Relation \refE{L8.2_2} obviously descends
to the following relation:
\begin{equation}\label{E:L8.2_2''}
    N^T\l +\frac{1}{2}N^T\til{A}N\equiv 0 ({\rm mod}\, 2\pi{\rm i}\Z) ,\quad \forall N\in\Z^g.
\end{equation}
Observe that $\til{A}=2diag(\l)$. Hence $\frac{1}{2}N^T\til{A}N=\sum_{k=1}^{g}\l_kn_k^2$ where $N=\sum_{k=1}^{g}n_ke_k$, and $N^T\l=\sum_{k=1}^{g}\l_kn_k$. Then we have
\begin{equation}\label{E:L8.2_2'''}
  N^T\l+\frac{1}{2}N^T\til{A}N = \sum_{k=1}^{g}\l_k n_k(n_k+1) \in 2\pi {\rm i}\Z
\end{equation}
because $\l_k\in \pi{\rm i}\Z$, and $n_k(n_k+1)\in 2\Z$ for every $k=1,\ldots,g$.
\end{proof}
\begin{remark}
In \cite{Dubr_ItNa} claim $1^\circ$ of the Lemma is given in the form which in our notation reads as $\theta(\til{\tau} z)=\ovl{\theta(z)}$. It immediately follows from our form of the claim due to relation \refE{tau_coord} and to the fact that $\theta$ is even in $z$.
\end{remark}
\begin{remark}\label{R:sym2}
In \cite{Dubr_ItNa}, $\l$ is given without any factor $\pi{\rm i}$, like in \cite{Dubr_Nat}, by mistake, because the normalizations of basis differentials (the expressions for $\theta$, resp.) are different in \cite{Dubr_Nat} and \cite{Dubr_ItNa}.
\end{remark}

%%%%%%%%%%%%%%%%%%%%%%%%%%%%%%%%%%%%%%%%%%
%%%%%%%%%%%%%%%%%%%%%%%%%%%%%%%%%%%%%%%%%%%%%%%
\section{Real curves with involutions, their Prym varieties, and $\theta$-functions}\label{S:ogr}

By a real curve with an involution we mean a compact nonsingular algebraic curve $\Sigma$ over $\C$ endowed with two commuting involutions $\s$ and $\tau$ where the first is holomorphic, while the second is antiholomorphic. We will refer to $\tau$ as to antiinvolution also. Let $g=g(\Sigma)$ denote the genus of $\Sigma$, $g_\s=g(\Sigma_\s)$ where $\Sigma_\s=\Sigma/\s$.
%%%%%%%%%%%%%%%%%%%%%%%%%%%%%%%%%%%%%%
\subsection{Prym matrix for a curve with an involution}
%%%%%%%%%%%%%%%%%%%%%%%%%%%%%%%%%%%%%%
We relax the requirement of realness for the curves considered here. Let the degree of the branching divisor of the natural covering $\Sigma\to\Sigma/\s$ be equal to $2n$. According to \cite{Fay}, there exists a $\s$-invariant base of cycles $a_i,b_i$ ($i=1,\ldots, g_\s$), $a_i,b_i$ ($i=g_\s+1,\ldots, h=g_\s+n-1$), $a_{i+h},b_{i+h}$ ($i=1,\ldots, g_\s$) on $\Sigma$, where the first and the third groups of cycles are pulled back from $\Sigma_\s$, such that the permutation of indices induced by the action of $\s$ on the elements of the base has the form $s=(1,h+1)\ldots(g_\s,h+g_\s)(g_\s+1)\ldots(h)$, and
\begin{equation}\label{E:cycles}
    \s(a_j)+a_{s(j)}= \s(b_j)+b_{s(j)}=0,\ j=1,\ldots g.
\end{equation}

Let $\{ w_i|  i=1,\ldots, g \}$ be the dual base of normalized holomorphic differentials on $\Sigma$, $w=(w_1,\ldots,w_g)^T$. Then
\begin{equation}\label{E:sigma*}
  \s^*w=-sw
\end{equation}
where $s$ written to the left of the column of differentials denotes the matrix of the linear transformation permuting coordinates according to the permutation $s$.
Indeed, $\int_{a_j}\s^*w=\int_{\s(a_j)}w$ by change of variables, $\int_{\s(a_j)}w=\int_{-a_{s(j)}}w = -\int_{a_{s(j)}}w$ by \refE{cycles}, and $\int_{a_{s(j)}}w =\int_{a_j}sw$ by symmetry of the matrix of $a$-periods.

Relations \refE{sigma*} can be also written in the form
\begin{equation}\label{E:sigma*'}
\s^*w_i=
\left\{
  \begin{array}{ll}
    -w_{i+h}, & \hbox{$i=1,\ldots g_\s$;} \\
    -w_i, & \hbox{$i=g_\s+1,\ldots,h$;}   \\
    -w_{i-h}, & \hbox{$i=h+1,\ldots,g$,}
  \end{array}
\right.
\end{equation}
Differentials $\{\w_i=w_i+w_{i+h} | i =1,\ldots g_\s\}$ and $\{ \w_i=w_i|i=g_\s+1,\ldots,h\}$ form a base of Prym differentials on $\Sigma$. This base is normalized in a sense that $\oint_{a_j}\w_k=2\pi i\d_{kj}$, $k,j=1,\ldots,h$.
We define the Prym matrix $\Pi=(\Pi_{ij})_{i,j=1,\ldots,h}$ as follows:
\begin{equation}\label{E:Prym_matr}
 \Pi_{ij}=\oint_{b_j}\w_i\ (j=1,\ldots,g_\s);\quad
 \Pi_{ij}=\frac{1}{2}\oint_{b_j}\w_i\ (j=g_\s+1,\ldots,h).
\end{equation}
Following Fay \cite{Fay} we often use the letters $\a,\b$ to denote indices $1,\ldots,g_\s$, and $i,j$ to denote $g_\s+1,\ldots,h$, and set $\a'=\a+h,\b'=\b+h$.
With this notation
{\renewcommand{\arraystretch}{2}
%\extrarowheight4pt
\begin{equation}\label{E:Prym_matr_Fay}
\Pi =
\left(\begin{array}{c|c}
        \Pi_{\a\b}  &  \Pi_{\a j}    \\
        \hline
        \Pi_{i\b} &  \Pi_{ij}   \\
        \end{array}\right) =
\left(\begin{array}{c|c}
        \int_{b_\b}\w_\a &  \frac{1}{2}\int_{b_j}\w_\a    \\
        \hline
        \int_{b_\b}\w_i &  \frac{1}{2}\int_{b_j}\w_i   \\
        \end{array}\right)
\end{equation}}
(cf. \cite[Eq. (92)]{Fay}). The expression of the Prym matrix via the Riemann matrix is given by
\begin{lemma}\label{L:Pi_via_B}
\begin{itemize}
\item[$1^\circ$.]
   $\Pi_{\a\b} = B_{\a\b} + B_{\a'\b} = \Pi_{\b\a} $;
\item[$2^\circ$.]
   $\Pi_{\a j} = \frac{1}{2}(B_{\a j}+B_{\a' j}) = \Pi_{j\a}$;
\item[$3^\circ$.]
   $\Pi_{ij} = \frac{1}{2}B_{ij} = \Pi_{ji} $.
\end{itemize}
\end{lemma}
\begin{proof}
$1^\circ$. By definition of $\Pi_{\a\b}$ and symmetry of $B$ we have $\Pi_{\a\b} = \int_{b_\b}(w_\a+w_{\a'}) = B_{\a\b}+B_{\a'\b} = B_{\b\a}+B_{\b\a'}$. Next, $B_{\b\a'}=B_{\b'\a}$. Indeed, by definition, relations \refE{cycles}--\refE{sigma*'} and change of variables  $B_{\b\a'} = \int_{b_{\a'}}w_{\b} = -\int_{\s(b_\a)}w_\b = -\int_{b_\a}\s^*w_\b = \int_{b_\a}w_{\b'} = B_{\b'\a}$. Finally $\Pi_{\a\b} = B_{\b\a}+B_{\b'\a} = \Pi_{\b\a}$.

$2^\circ$. By definition $\Pi_{\a j} = \frac{1}{2}(B_{\a j}+B_{\a' j})$, $\Pi_{j\a}=B_{j\a}$. But $B_{\a j} = \int_{b_j}w_\a = \int_{-b_j}\s^*w_\a = \int_{b_j}w_{\a'}=B_{\a' j}$. Hence $\Pi_{\a j} = B_{\a j} = B_{j\a} = \Pi_{j\a}$.

$3^\circ$. The proof of the claim is similar to the proof two above ones.
\end{proof}

%%%%%%%%%%%%%%%%%%%%%%%%%%%%%%%%%%%%%%%%%%
\subsection{Symmetries of the Prym matrix for a real curve with an involution}\label{SS:Prym_rsym}
As above (\refSS{realJac}), let $t$ denote the permutation of indices induced by $\tau$. Since $\s$ and $\tau$ commute, there exists a real base of cycles satifying conditions \refE{cycles}.
\begin{lemma}\label{L:sym_prym}
For separating real curves with an involution, the matrix $\Pi$ satisfies to the following equivalent conditions:
\begin{itemize}
  \item[$1^\circ$.]
     $\ovl{\Pi_{pq}}=\Pi_{t(p),t(q)}$, $p,q=1,\ldots,h$;
  \item[$2^\circ$.]
     $\ovl{\Pi}=t\Pi t$ where, in abuse of notation, the matrix $t$ is given by the permutation $t$: $t_{pq}=\d_{p,t(q)}$;
  \item[$3^\circ$.]
     $\ovl{\Pi_q}=t\Pi_{t(q)}$ where $\Pi_q$ is the $q$-th column of $\Pi$.
\end{itemize}
\end{lemma}
\begin{proof}
For separating curves, \refL{8.2} can be summarized as follows: $\ovl{B_{pq}}=B_{t(p),t(q)}$, $p,q=1,\ldots,g$. Making use of that we prove the claim $1^\circ$ of the present lemma for every block of the matrix \refE{Prym_matr_Fay}, separately.

Indeed, by \refE{Prym_matr_Fay} $\Pi_{\a\b}=B_{\a\b}+B_{\a'\b}$, hence $\ovl{\Pi_{\a\b}}=\ovl{B_{\a\b}}+\ovl{B_{\a'\b}} = B_{t(\a),t(\b)}+B_{t(\a'),t(\b)} = B_{t(\a),t(\b)}+B_{t(\a)',t(\b)} = \Pi_{t(\a),t(\b)}$ ($t(\a')=t(\a)'$ by commutativity of $\s$ and $\tau$). Similarly
\begin{align*}
  & \ovl{\Pi_{i\b}}= \ovl{B_{i\b}} = B_{t(i),t(\b)} = \Pi_{t(i),t(\b)}, \\
  & \ovl{\Pi_{\a j}}= {\textstyle\frac{1}{2}}\ovl{B_{\a j}} = {\textstyle\frac{1}{2}}B_{t(\a), t(j)} = \Pi_{t(\a), t(j)},\\
  & \ovl{\Pi_{i j}}= {\textstyle\frac{1}{2}}\ovl{B_{i j}} = {\textstyle\frac{1}{2}}B_{t(i), t(j)} = \Pi_{t(i), t(j)}.
\end{align*}
This completes the proof of the claim $1^\circ$. The claims $2^\circ$ and $3^\circ$ are obviously equivalent to the claim $1^\circ$.
\end{proof}
Now assume the curve to be non-separating. Assume $a_1,\ldots,a_{r_0}$ and $a_{g_\s+1},\ldots,a_{g_\s+r_1}$ ($r_0\le g_\s$, $r_1\le n-1$) are ovals, and the other $a$-cycles are not.
\begin{lemma}\label{L:8.2_Prym}
Let $\{ a_j,b_j\}$ be a $\s$-invariant real base of cycles of a non-separating real curve, then
\begin{itemize}
\item[$1^\circ$.]
$\overline{\Pi_\b}=\Pi_\b$ for $\b=1,\ldots,r_0$ (i.e. for $\b$  corresponding to pairs of ovals permutable by $\s$);
\item[$2^\circ$.]
$\overline{\Pi_j}=\Pi_j$ for $j=g_\s+1,\ldots,g_\s+r_1$ (i.e. for $j$ corresponding to $\s$-invariant ovals);
\item[$3^\circ$.]
$\overline{\Pi_\b}=\Pi_\b-A_\b$ for  $\b= r_0+1,\ldots,g_\s$ (i.e. for non-ovals permutable by $\s$);
\item[$4^\circ$.]
$\overline{\Pi_j}=\Pi_j-\frac{1}{2}A_j$ for $j= g_\s+r_1+1,\ldots,h$ (i.e. for $\s$-invariant non-ovals).
\end{itemize}
\end{lemma}
\begin{proof}
$1^\circ$. For $\Pi_\b$, $\b\le r_0\le g_\s$ we have
\[
 \begin{array}{ll}
   \Pi_{\a\b} = B_{\a\b}+B_{\a'\b}& \hbox{by \refL{Pi_via_B};} \\
   \Pi_{i\b} = B_{i\b} & \hbox{by \refE{Prym_matr_Fay}.}
 \end{array}
\]
Since number $\b$ corresponds to an oval, we obtain by \refL{8.2} that $B_{\a\b}=\ovl{B_{\a\b}}$, $B_{\a\b}=\ovl{B_{\a'\b}}$, $B_{i\b}=\ovl{B_{i\b}}$. Hence $\ovl{\Pi_\b}=\Pi_\b$.

$2^\circ$. For $\Pi_j$, $g_\s+1\le j\le g_\s+r_1$ the argument is the same, except for the coefficient $\frac{1}{2}$ in \refL{Pi_via_B}.$2^\circ$, which does not disturb.

$3^\circ$. For $\a\le g_\s$ it follows from $\Pi_{\a\b}= B_{\a\b}+B_{\a'\b}$  and \refL{8.2},$2^\circ$. Observe that $A_{\a'\b}=2\pi{\rm i\d_{\a'\b}}=0$ since $\a\le g_\s$, $\a' > h\ge g_\s$. Hence $\Pi_{\a\b}-A_{\a\b}= B_{\a\b}-A_{\a\b} + B_{\a'\b}- A_{\a'\b}$. According to \refL{8.2},$2^\circ$ the last is equal to $\ovl{B_{\a\b}} + \ovl{B_{\a'\b}} = \ovl{\Pi_{\a\b}}$.

For $i> g_\s$, $\Pi_{i\b}=B_{i\b}$ (like in the point $1^\circ$), hence $\Pi_{i\b}-A_{i\b} = B_{i\b} - A_{i\b} = \ovl{B_{i\b}} = \ovl{\Pi_{i\b}}$.

$4^\circ$. For $\a\le g_\s$ the claim follows from $\Pi_{\a j}=\frac{1}{2}(B_{\a j}+B_{\a' j})$ and \refL{8.2}.$2^\circ$. In this case $A_{\a j}=A_{\a' j}=0$ since $\a\le g_\s<j\le h<\a'$. Hence $\Pi_{\a j}-\frac{1}{2}A_{\a j}= \frac{1}{2}(B_{\a j}-A_{\a j} + B_{\a' j}- A_{\a' j})$. According to \refL{8.2},$2^\circ$ the last is equal to $\frac{1}{2}(\ovl{B_{\a j}}+\ovl{B_{\a' j}}) = \ovl{\Pi_{\a j}}$. In fact, we have obtained $\ovl{\Pi_{\a j}} = \Pi_{\a j}$ in this case.

For $i> g_\s$ the claim follows from $\Pi_{ij}=\frac{1}{2}B_{ij}$ (\refL{Pi_via_B}.$3^\circ$). Indeed, $\Pi_{ij}-\frac{1}{2}A_{ij}= \frac{1}{2}(B_{ij} - A_{ij}) =\frac{1}{2}\ovl{B_{ij}} = \ovl{\Pi_{ij}}$.
\end{proof}
\begin{corollary}
All entries of $\Pi$ are real numbers except for diagonal entries $\Pi_{jj}$ with  $j= r_0+1,\ldots,g_\s$ and $j= g_\s+r_1+1,\ldots,h$ (for Jacobians, except for the entries of  the right lower corner).
\end{corollary}
%%%%%%%%%%%%%%%%%%%%%%%%%%%%%%%%%%%%%%

%%%%%%%%%%%%%%%%%%%%%%%%%%%%%%%%%%%%%%%%%%%%%%
\subsection{Prym theta function, isoPrym variety and the Abel--Prym map}\label{SS:Ab_Pr}
%%%%%%%%%%%%%%%%%%%%%%%%%%%%%%%%%%%%%%%%%%%%%%
A Prym theta function is a Riemann $\theta$ function with the Riemann matrix $\Pi$:
\begin{equation}\label{E:theta_Prym}
  \theta(z,\Pi)=\sum_{N\in \Z^h}\exp(\frac{1}{2}N^T\Pi N+N^Tz),\quad z\in\C^h.
\end{equation}
It corresponds to the principally polarized Abelian variety $P_0=\C^{h}/\Z({2\pi i}E,\Pi)$ which is a finite unramified covering of the Prym variety (thus isogeneous to the Prymian). Below, $P_0$ is referred to as isoPrymian.  The map $\A : \Sigma\to P_0$:
\begin{equation}\label{E:AP_map}
  \A(\ga)= \left(\int_{\ga_0}^{\ga}\w\right)\,({\rm mod}\,\Z({2\pi i}E,\Pi))
\end{equation}
where $\w=(\w_1,\ldots,\w_{h})^T$ (equivalently, $\w=(\w_\a,\w_j)^T$) is reffered to as the Abel--Prym map.
\begin{lemma}\label{L:P_symm}
The Prym $\theta$ function of a real curve possesses the following symmetries:
\begin{itemize}
  \item[$1^\circ$]
$\ovl{\theta(z)}=\theta(t\ovl{z})$ for separating curves;
  \item[$2^\circ$]
For non-separating curves, let $\l=(\l_1,\ldots,\l_h)^T$ where
\begin{equation}\label{E:lambda_Prym}
\l_k=
\left\{
              \begin{array}{ll}
                \pi{\rm i}, & \hbox{$k=r_0+1,\ldots,g_\s$;} \\
     \frac{1}{2}\pi{\rm i}, & \hbox{$k=g_\s+r_1+1,\ldots,h$;}\\
                         0  & \hbox{otherwise.}
              \end{array}
            \right.
\end{equation}
If all $\s$-invariant basis cycles are ovals (i.e. $g_\s+r_1=h$) then
\[
   \ovl{\theta(z)}=\theta(\ovl{z}+\l).
\]
\end{itemize}
\end{lemma}
\begin{proof}
$1^\circ$. Due to \refL{sym_prym},$2^\circ$ the proof of the claim $1^\circ$ of the present lemma is the same as the proof of \refL{symm},$1^\circ$.

$2^\circ$.
Similar to \refL{symm}, we can summarize \refL{8.2_Prym} as $\ovl{\Pi}=\Pi-\til{A}$ where $\til{A}=2diag(\l)$. Like in \refL{symm}, the desired symmetry reduces to the relation
\begin{equation}\label{E:sym_r}
    N^T\l +\frac{1}{2}N^T\til{A}N\equiv 0 ({\rm mod}\, 2\pi{\rm i}\Z) ,\quad \forall N\in\Z^g,
\end{equation}
and
\begin{equation}\label{E:sym_r}
  N^T\l+\frac{1}{2}N^T\til{A}N = \sum_{k=1}^{g}\l_k n_k(n_k+1).
\end{equation}
However, it follows from \refE{lambda_Prym} that for $k=g_\s+r_1+1,\ldots,h$ (i.e. for $\s$-invariant non-ovals) we only can claim that $\l_k n_k(n_k+1)\in \pi{\rm i}\Z$ while for the remainder of values of $k$ we have $\l_k n_k(n_k+1)\in 2\pi{\rm i}\Z$. Hence in absence of such non-ovals the desired symmetry takes place.
\end{proof}
%%%%%%%%%%%%%%%%%%%%%%%%%%%%%%%%%%%%%%%%%%%%%%%%%%%%%%%%%%%%

%%%%%%%%%%%%%%%%%%%%%%%%%%%%%%%%%%%%%%%%%%%%%%%%%%%%%

%%%%%%%%%%%%%%%%%%%%%%%%%%%%%%%%%%%%%%%%%%%%%%%%%%%%%
%%%%%%%%%%%%%%%%%%%%%%%%%%%%%%%%%%%%%%%%%%%%%%%%%%%%%
\section{The Abel--Prym map, and inversion theorem}\label{S:inv_th}
%%%%%%%%%%%%%%%%%%%%%%%%%%%%%%%%%%%%%%%%%%%%%%%%%%
%%%%%%%%%%%%%%%%%%%%%%%%%%%%%%%%%%%%%%%%%

\subsection{Riemann vanishing theorem}
Assume $\Sigma$ to be endowed with a holomoprhic involution~$\s$. Let $A$ stay for the Abel map, and $A'$ for the Abel--Prym map.
Let $F_{z}(P)=\theta(\int_{q}^{P}\w-z)=\theta(A'(P)-z)$ ($P\in\Sigma$, $z\in \C^h$)  where $\theta$ is the Prym $\theta$-function. The $F_{z}(P)$ is defined up to a $\theta$-multiplier depending on the integration path from $q$ to~$P$, hence the zero divisor of $F_{z}(P)$ is well defined. We denote it by $\zeta$.  We introduce the matrix $\eps=diag(\eps_1,\ldots,\eps_h)$ where $\eps_\a=1$ for $\a=1,\ldots,g_\s$, $\eps_j=2$ for $j=g_\s+1,\ldots,h$. Following \cite{Fay} we also introduce the embedding $\phi:\, \C^h\to \C^g$ by setting $\phi(z_1,\ldots,z_{g_\s},z_{g_\s+1},\ldots,z_h)\to (z_1,\ldots,z_{g_\s},2z_{g_\s+1},\ldots,2z_h,z_1,\ldots,z_{g_\s})$.  Let $\sum_{i=1}^{2n}Q_i$ be the branch divisor of the natural covering $\pi:\,\Sigma\to\Sigma_\s$,  $K_\s$ be the canonical divisor on~$\Sigma_\s$.
\begin{lemma}\label{L:inver1}
Either $F_{z}(P)\equiv 0$ or
\begin{itemize}
\item[$1^\circ$]
$\deg\zeta=2h$;
\item[$2^\circ$]
$A(\zeta)= \phi(z)+ \Delta$, $\Delta=\sum_{i=1}^{2n}A(Q_i)+\pi^*(\frac{1}{2}K_\s)+q-\s q$;
\item[$3^\circ$]
$\pi_*\zeta$ is the divisor of zeroes of a differential on $\Sigma_\s$ with at most simple poles at $\pi_*(\sum_{i=1}^{2n}Q_i)$: $\pi_*\zeta\sim \sum_{i=1}^{2n}\pi_*(Q_i)+K_\s$.
\item[$4^\circ$]
$A'(\zeta)=\eps z$.
\end{itemize}
\end{lemma}
Claims $1^\circ$--$3^\circ$ are nothing but Corollary 5.6 \cite{Fay} combined with Eq.(108) \cite{Fay}. Claim $4^\circ$ follows from Lemma 2.5 \cite{Sh_PrymJacobi} which states that $A'(\zeta)=\eps z+\Delta'$ where $\Delta'\in P_0$ does not depend of $z$. It is completely similar to the claim $2^\circ$ of the present lemma. However, observe that $A'(\s\zeta)=-A'(\zeta)$ for any $\zeta$, and $\s\zeta$ is obviously the divisor of zeroes of $\theta(A'(\s P)-z)$. We obtain $\theta(A'(\s P)-z) = \theta(-A'(P)-z) = \theta(A'(P)+z)$ since $\theta$ is an even function. Hence $\s\zeta$ is the zero divisor of the function $\theta(A'(P)+z)$, and by Lemma~2.5 \cite{Sh_PrymJacobi} $A'(\s\zeta) = -\eps z+\Delta'$ which implies $A'(\zeta)+A'(\s\zeta)=2\Delta'$. On the other hand,  $A'(\zeta)+A'(\s\zeta)=0$ by skew-symmetry of $A'$, hence $\Delta'=0$.
\begin{remark}
Skew-symmetry of $A'(P)$ with respect to $\s$ is obvious if $q$ is $\s$-invariant. If there is no $\s$-invariant points on $\Sigma$ ($\pi$ is non-ramified) we can slightly modify the definition of $A'$ according to \cite[section 4.2]{BorSh}, or to \cite[section 7.2.1]{Taim}. The proof of Lemma 2.5 \cite{Sh_PrymJacobi} remains valid.
\end{remark}
\begin{corollary}\label{C:Nov-Ves}
\begin{equation}\label{E:Nov-Ves}
          A(\zeta)+A(\s\zeta)=2\Delta.
\end{equation}
\end{corollary}
Relation \refE{Nov-Ves} is referred to as Veselov--Novikov condition \cite{VN}.
\begin{proof}
Let $z\in \C^h$ and $\zeta$ is the divisor of the function $\theta(A'(P)-z)$. Then, as above, $\s\zeta$ is the divisor of $\theta(A'(P)+z)$, hence $A(\s\zeta)=\phi(-z)+\Delta$ by \refL{inver1}.$2^\circ$. It follows that  $A(\zeta)+A(\s\zeta)= \phi(z)+ \phi(-z)+ 2\Delta$. But $\phi$ is obviously skew-symmetric by definition.
\end{proof}%%%%%%%%%%%%%%%%%%%%%%%%%%%%%%%%%%%%%%%%%%%%%%%
\subsection{$\theta$-function formula for symmetric functions of zeroes of $F_{\eps^{-1}z}(P)$} \label{S:theta-form}

In general, any effective procedure of finding out the divisor $\zeta$ is unknown. However, there exists an effective (avoiding any direct solution of the equation $\theta(A(P)-\eps^{-1}z)=0$) procedure of calculating symmetric functions of the points of the divisor $\zeta$. Here we present it following the lines of \cite{Sh_PrymJacobi}, with modifications due to the difference between the particular case considered there, and the general case studied here. In turn, the approach in \cite{Sh_PrymJacobi} developes the one proposed in \cite{Dubr_theta}, and finally goes back to Riemann.

Let $z\in isoPrym(\Sigma)$, $\A^{-1}(z)=\zeta$, and $|\zeta| = \{ P_1,\ldots,P_{2h} \}$ is the support of~$\zeta$. Symmetric functions of  $P_1,\ldots,P_{2h}$ are well-defined functions of $z$. It is our goal to find out a theta function formulae for a certain set of such functions.

For any meromorphic function $f$ on $\Sigma$ let $\s_f(z)=\sum_{P\in\zeta} f(P)$. Below, we assume that $\Sigma$ is a branch covering of the Riemann sphere, and $f$ has no pole except those over infinity. Let $\pi:\Sigma\to\mathbb{P}^1$ be the covering map. Then we have the following relation close to the relation due to Dubrovin (\cite[Eq. (11.23)]{Dubr_RegCh}, \cite[Eq. (2.4.29)]{Dubr_theta})  (see the proof in \cite{Sh_PrymJacobi}):
\begin{equation}\label{E:Dubr}
  \s_f(z)=c-\sum_{q\in\pi^{-1}(\infty)}\res_q fd\ln F_{\eps^{-1}z},
\end{equation}
where $c$ is constant in $z$.

Let $x_i=\pi(P_i)$, $i=1,\ldots,2h$. We take $f(P)=x^k$ where $x=\pi(P)$. Denote $\s_f$ by $\s_k$, then
\begin{equation}\label{E:sigm}
  \s_k(z)=x_1^k+\ldots +x_{2h}^k,
\end{equation}
i.e. $\s_k(z)$ is the $k$th Newton polynomial in $x_1,\ldots ,x_{2h}$.
Exactly as in \cite{Sh_PrymJacobi} we obtain
\begin{equation}\label{E:sigma_k_result}
    \s_k(z)=c-\sum_{q\in\pi^{-1}(\infty)}\sum_{i=1}^{h}\sum_{1\le |j|\le 2k-1} \varkappa_{qik}^jD^j\partial_i\ln\theta(\A(q)-\epsilon^{-1}z)
\end{equation}
where $j=(j_1,\ldots,j_{h_1})$, $|j|=j_1+\ldots+j_{h_1}$,
\begin{equation}\label{E:theta_kappa}
D^j=\frac{1}{j_1!\ldots j_h!}\frac{\partial^{|j|}}{\partial z_1^{j_1}\ldots\partial z_{h}^{j_{h}}},
  \quad {\varkappa_{qik}^j}=\sum_{l_{qi}+\sum_{s=1}^{h}\sum_{p=1}^{j_s} l_{qsp}=2k-1} \varphi_{qi}^{(l_{qi})}\prod_{s=1}^{h}\prod_{p=1}^{j_s}\frac{\varphi_{qs}^{(l_{qsp}-1)}}{l_{qsp}},
\end{equation}
$l_{qs}$ and $\varphi_{qs}^{(l_{qs})}$ are defined from the relation $\A_s(P)=\A_s(q)+\sum_{l_{qs}\ge 1}\frac{\varphi_{qs}^{(l_{qs})}}{l_{qs}}z_q^{l_{qs}}$, $P=P(z_q)$, $z_q$ is a local parameter at the point $q$.

The functions $\s_k(z)$, $k=1,\ldots,2h$ give a full set of symmetric functions of $x$-coordinates of the points in $A^{-1}(z)$. They determine $x_1,\ldots,x_{2h}$ up to a permutation. With this relation, observe that, for hyperellyptic curves, there exists an alternative to the Rieman vanishing theorem \cite{BEL}. Namely, the $x_i$'s can be calculated as zeroes of an algebraic equation with coefficients given by means of the Weierstra{\ss} $\wp$-functions on the curve in that case. The following statement shows that knowledge of $\s_k(z)$, $k=1,\ldots,2h$ gives the solution of the Jacobi inversion problem of the same degree of effectiveness in the case of an arbitrary curve with involution.
\begin{theorem}\label{T:qusi_baker}
 Calculating $\s_k(z)$, $k=1,\ldots,2h$ reduces the problem of finding out the divisor $\zeta$ to resolving a degree $2h$ algebraic equation.
\end{theorem}
Indeed, we pass from the Newtone polynomials to elementary symmetric functions in \refE{sigm}, and then find out $x_1,\ldots,x_{2h}$ as the roots of the corresponding algebraic equation.
%%%%%%%%%%%%%%%%%%%%%%%%%%%%%%%%%%%%%%%%%%%%%%

%%%%%%%%%%%%%%%%%%%%%%%%%%%%%%%%%%%%%%%%%%%%%%%%%%
\subsection{Inversion theorem for real curves}

It is a purpose of this section to show that for curves possessing a real structure $\tau$, the inverse image of certain real subvarieties of the isoPrymian (actually being shifts of varieties of fixed points of symmetries of the Prym theta function) can be given a more effective description as the variety of $\tau$-invariant (or $\s\tau$-invariant) degree $2h$ divisors (automatically satisfying the Veselov--Novikov condition by \refC{Nov-Ves}). Observe that both $\tau$ and $\s\tau$ are antiholomorphic involutions on $\Sigma$, and both of them are involved into the description.
The following statement completes the inversion theorem for real curves.
\begin{theorem}\label{T:invers}
Assume, $F_{\eps^{-1}z}(P)$ is not identically zero. Then for separating curves
\begin{itemize}
  \item[$1^\circ$.] if $z+t\ovl{z}= 0$ then $\zeta=A^{-1}(z)$ is $\tau$-invariant;
  \item[$2^\circ$.] if $z-t\ovl{z}= 0$ then $\zeta=A^{-1}(z)$ is $\s\tau$-invariant;
\end{itemize}
and for non-separating curves
\begin{itemize}
  \item[$3^\circ$.] if $z-\ovl{z}= -\eps\l$ then $\zeta=A^{-1}(z)$ is $\s\tau$-invariant.
\end{itemize}
\end{theorem}
\begin{proof}
By \refL{inver1} the zero divisor of $F_{\eps^{-1}z}(P)$, and $A^{-1}(z)$ are the same. Consider first the case of separation curves.

$1^\circ$. Let's find out when the zero divisor of $F_{\eps^{-1}z}(P)$ is $\tau$-invariant. It takes place if $F_{\eps^{-1}z}(\tau(P))=\ovl{F_{\eps^{-1}z}(P)}$ (up to a $\theta$-multiplier depending on the integration path).
By \refL{tau^*}.$1^\circ$ $A(\tau(P)) = -t\ovl{A(P)}$. Hence $F_{\eps^{-1}z}(\tau(P))  = \theta(-t\ovl{A(P)}-\eps^{-1}z) = \theta(t\ovl{A(P)}+\eps^{-1}z)$ (the last because $\theta$ is an even function). By \refL{P_symm},$1^\circ$
\[
    \ovl{F_{\eps^{-1}z}(P)}  = \ovl{\theta(A(P)-\eps^{-1}z)} = \theta(t(\ovl{A(P)}-\eps^{-1}\ovl{z})).
\]
If $t\ovl{A(P)}+\eps^{-1}z = t(\ovl{A(P)}-\eps^{-1}\ovl{z})$ then $F_{\eps^{-1}z}(\tau(P))  = \ovl{F_{\eps^{-1}z}(P)}$. Since $\eps$ commutes with $t$, the last implies the claim $1^\circ$ of the lemma.

$2^\circ$.   By \refL{P_symm}.$1^\circ$, and due to the fact that $\w$ in \refE{AP_map} is $\s$-antiinvariant, we have $A(\s\tau(P))=t\ovl{A(P)}$. Hence $F_{\eps^{-1}z}(\s\tau(P))  = \theta(t\ovl{A(P)}-\eps^{-1}z)$. By \refL{P_symm}.$1^\circ$
\[
    \ovl{F_{\eps^{-1}}(P)}  = \ovl{\theta(A(P)-\eps^{-1}z)} = \theta(t(\ovl{A(P)}-\eps^{-1}\ovl{z})).
\]
If $t\ovl{A(P)}-\eps^{-1}z = t(\ovl{A(P)}-\eps^{-1}\ovl{z})$ then $F_{\eps^{-1}z}(\s\tau(P))  = \ovl{F_{\eps^{-1}z}(P)}$. Since $\eps$ commutes with $t$, the last implies the claim $2^\circ$ of the lemma.

$3^\circ$. Non-separation curves.

By \refL{tau^*}.$3^\circ$, and due to the fact that $\w$ in \refE{AP_map} is $\s$-antiinvariant, $A(\s\tau(P)) = \ovl{A(P)}$. Hence $F_{\eps^{-1}z}(\s\tau(P))  = \theta(\ovl{A(P)}-\eps^{-1}z)$. By \refL{P_symm}.$2^\circ$
\[
    \ovl{F_{\eps^{-1}z}(P)}  = \ovl{\theta(A(P)-\eps^{-1}z)} = \theta(\ovl{A(P)}-\eps^{-1}\ovl{z}+\l).
\]
If $\ovl{A(P)}-\eps^{-1}z = \ovl{A(P)}-\eps^{-1}\ovl{z}+\l$ then $F_{\eps^{-1}z}(\s\tau(P))  = \ovl{F_{\eps^{-1}z}(P)}$. The last implies the claim $3^\circ$ of the lemma.
\end{proof}
\begin{remark}
For non-separating curves the case when $\zeta=A^{-1}(z)$ is $\tau$-invariant does not exist, for the reason that it descends to the relation $z +\ovl{z} = -\l$ which is contradictory because $\l$ is imaginary, non-zero.
\end{remark}

%%%%%%%%%%%%%%%%%%%%%%%%%%%%%%%%%%%%%%%%%%%%%%%%%%%%

%%%%%%%%%%%%%%%%%%%%%%%%%%%%%%%%%%%%%%%%%%%%%%%
%%%%%%%%%%%%%%%%%%%%%%%%%%%%%%%%%%%%%%%%%%%%%%%
%%%%%%%%%%%%%%%%%%%%%%%%%%%%%%%%%%%%%%%%%%%%%%%
%%%%%%%%%%%%%%%%%%%%%%%%%%%%%%%%%%%%%%%%%%%%%%%%%%%%

\bibliographystyle{amsalpha}

\end{document}
%%%%%%%%%%%%%%%%%%%%%%%%%%%%%%%%%%%%%%%%%%%%%%%%
%%   THE END
%%%%%%%%%%%%%%%%%%%%%%%%%%%%%%%%%%%%%%%%%%%